\documentclass[12pt,reqno]{article}
\usepackage{hyperref}
\usepackage{empheq}
\newcommand*{\tran}{^{\mkern-1.5mu\mathsf{T}}}
\usepackage{enumerate}
\usepackage{paralist}
\usepackage[doc]{optional}
\usepackage{xcolor}
\usepackage{float}
\usepackage{soul}
\usepackage{graphicx}
\definecolor{labelkey}{rgb}{0,0.08,0.45}
\definecolor{refkey}{rgb}{0,0.6,0.0}
\definecolor{Brown}{rgb}{0.45,0.0,0.05}
\definecolor{lime}{rgb}{0.00,0.8,0.0}
\definecolor{lblue}{rgb}{0.5,0.5,0.99}

 \usepackage{mathpazo}

\usepackage{amsmath}

\usepackage{stmaryrd}
\usepackage{amssymb}
\usepackage{amsthm}

\newcommand{\nnn}{\ensuremath{{n\in{\mathbb N}}}}
\newcommand{\kkk}{\ensuremath{{k\in{\mathbb N}}}}

\newcommand{\menge}[2]{\big\{{#1}~\big |~{#2}\big\}}

\newcommand{\To}{\ensuremath{\rightrightarrows}}

\newcommand{\fenv}[1]%
{\ensuremath{\,\overrightarrow{\operatorname{env}}_{#1}}}
\newcommand{\benv}[1]%
{\ensuremath{\,\overleftarrow{\operatorname{env}}_{#1}}}

\newcommand{\scal}[2]{\left\langle{#1},{#2}  \right\rangle}

\newcommand{\zeroun}{\ensuremath{\left]0,1\right[}}
\newcommand{\RR}{\ensuremath{\mathbb R}}
\newcommand{\SR}{\ensuremath{\mathbb S}}

\newcommand{\Id}{\ensuremath{\operatorname{Id}}}

\newcommand{\gconv}{\ensuremath{\,\overset{\mathsf{g}}{\rightarrow}}\,}
\newcommand{\pconv}{\ensuremath{\,\overset{\mathsf{p}}{\rightarrow}}\,}

{\begin{list}{}{%
\settowidth{\labelwidth}{\textrm{#1~}}%
\setlength{\leftmargin}{\labelwidth+\labelsep}}}
{\end{list}}

\newtheorem{theorem}{Theorem}[section]
\newtheorem{lemma}[theorem]{Lemma}

\newtheorem{corollary}[theorem]{Corollary}

\newtheorem{proposition}[theorem]{Proposition}

\newtheorem{example}[theorem]{Example}
\newtheorem{ex}[theorem]{Example}

\newtheorem{remark}[theorem]{Remark}

\providecommand{\norm}[1]{\lVert#1\rVert}

\providecommand{\innp}[1]{\langle#1\rangle}

\providecommand{\RR}{\mathbb{R}}

\providecommand{\gr}{\operatorname{gra}}

\providecommand{\Id}{\operatorname{{ Id}}}

\providecommand{\To}{\rightrightarrows}

\providecommand{\gr}{\operatorname{gra}}

\providecommand{\Id}{\operatorname{Id}}

\providecommand{\nnn}{{n\in\NN}}

\providecommand{\RR}{\mathbb{R}}

\providecommand{\mM}{\mathcal{M}}
\providecommand{\mL}{\mathcal{L}}
\providecommand{\mS}{\mathcal{S}}

\definecolor{myblue}{rgb}{.8, .8, 1}

\allowdisplaybreaks 

\begin{document}
%
\title{\textsc
On Douglas--Rachford operators\\ that fail to be proximal mappings}
\author{
Heinz H.\ Bauschke\thanks{
Mathematics, University
of British Columbia,
Kelowna, B.C.\ V1V~1V7, Canada. E-mail:
\texttt{heinz.bauschke@ubc.ca}.},~
Jason Schaad\thanks{
Department of Mathematics and Statistics,
Okanagan College, 1000 K.L.O.\ Road, 
Kelowna, B.C.\ V1Y~4X8, Canada. E-mail: 
\texttt{JSchaad@okanagan.bc.ca}.}, 
~and Xianfu Wang\thanks{
Mathematics, University of
British Columbia,
Kelowna, B.C.\ V1V~1V7, Canada. E-mail:
\texttt{shawn.wang@ubc.ca}.}}
\date{February 17, 2016}
\maketitle
\begin{abstract}
\noindent
The problem of finding a zero of the sum of two maximally
monotone operators is of central importance in optimization.
One successful method to find such a zero is the
Douglas--Rachford algorithm which iterates a firmly nonexpansive
operator constructed from the 
resolvents of the given monotone operators.

In the context of finding minimizers of convex functions, 
the resolvents are actually proximal mappings. 
Interestingly,
as pointed out by Eckstein in 1989, 
the Douglas--Rachford operator itself may fail to be a proximal
mapping. 
We consider the class of symmetric linear relations that are
maximally monotone and
prove the striking result that
the Douglas--Rachford operator is generically not a proximal
mapping. 
\end{abstract}
{\small

}
\section{Introduction}
Throughout this paper, we work in 
the standard Euclidean space 
\begin{equation}
X=\RR^n,
\end{equation}
equipped 
with the standard inner product $\innp{\cdot,\cdot}$ and
induced Euclidean norm $\norm{\cdot}$.
Recall that a set-valued operator
\begin{equation}
A\colon X\To X
\end{equation}
is \emph{monotone} if $\scal{x-y}{x^*-y^*}\geq 0$ 
whenever $(x,x^*)$ and $(y,y^*)$ belong
to $\gr A$, the graph of $A$; 
$A$ is \emph{maximally monotone} if any proper enlargement of $A$ fails to
be monotone.
Maximally monotone operators are of importance in modern
optimization (see 
\cite{AT}, 
\cite{BC2011}, 
\cite{Bor50}, 
\cite{BorVan},
\cite{BorZhu}, 
\cite{Bot}, 
\cite{Regina},
\cite{EckBer}, 
\cite{Simons})
as they cover subdifferential operators
of functions that are convex lower semicontinuous and proper as
well as matrices whose symmetric part is positive semidefinite.
A central problem is to 
\begin{equation}
\label{e:sum}
\text{find $x\in X$ such that $0\in Ax+Bx$,}
\end{equation}
where $A$ and $B$ are maximally monotone on $X$.
For instance, if $A=\partial f$ and $B=\partial g$,
where $f$ and $g$ belong to $\Gamma_0(X)$, the set of functions
that are convex, lower semicontinuous and proper on $X$, then the 
sum problem \eqref{e:sum} is tied to the problem of finding a
minimizer of $f+g$. 
A popular iterative method, dating back to Lions and Mercier's
seminal work \cite{LM}, to solve \eqref{e:sum} is 
the \emph{Douglas--Rachford algorithm} whose governing sequence
$(x_n)_\nnn$ is given by
\begin{equation}
\label{e:DR}
(\forall\nnn)\quad x_{n+1} = T_{A,B}x_n,
\end{equation} 
where
$T_{A,B}=\Id-J_A+J_BR_A = (\Id+R_BR_A)/2$ is the Douglas--Rachford splitting
operator, $J_A = (\Id+A)^{-1}$ is the \emph{resolvent} of $A$ and
and $R_A = 2J_A-\Id$ is the \emph{reflected resolvent}. 
If $Z$, the set of solutions \eqref{e:sum}, is nonempty,
then $(x_n)_\nnn$ converges to a fixed point of $T_{A,B}$ and
$(J_Ax_n)_\nnn$ converges to a point in $Z$. 
In fact, as pointed in \cite{LM}, one has $T_{A,B} = J_C$ for some
maximally monotone operator $C$ depending on $(A,B)$.
That is, \eqref{e:DR} is actually the iteration of a resolvent
--- the resulting method was carefully studied by Rockafellar
\cite{Rockprox}.
If the operator $C$ is actually a subdifferential operator, i.e.,
$C=\partial h$, where $h\in \Gamma_0(X)$; or equivalently if
$J_C$ is a \emph{proximal map} (a.k.a.\ proximity operator) \cite{Moreau}, 
then stronger
statements are available concerning the resolvent iteration \cite{Guler}.
This prompts interest in the question whether $C=\partial h$.
Unfortunately, in general, $T_{A,B}=J_C$ is only a resolvent, not
a proximal map as demonstrated by Eckstein \cite{Eckthesis}; 
the following simpler example is from Schaad's thesis 
\cite{Schaad}. 
Suppose that $X=\RR^2$, and that $A$ and $B$ are the normal cone
operators of the subspaces $\RR(1,0)$ and $\RR(1,1)$. 
Then the associated maximally monotone operator is given by the
matrix 
\begin{equation}
C = \left[
\begin{array}{cc}
0 & 1 \\
-1 & 0 
\end{array}
\right], 
\end{equation}
which is \emph{not} symmetric and hence $C$ is not a
subdifferential operator.
The corresponding Douglas--Rachford operator
\begin{equation}
J_C = T_{A,B} = 
\frac{1}{2}\left[
\begin{array}{cc}
1 & -1 \\
1 & 1 
\end{array}
\right], 
\end{equation}
which is also \emph{not} symmetric, 
is therefore only a resolvent but \emph{not} a proximal mapping. 
This is surprising because \eqref{e:sum} corresponds in this case
to the convex feasibility problem asking to find a point in 
$\RR(1,0)\cap \RR(1,1)$ (which is $\{(0,0)\}$).

\emph{In this note paper we 
show that in the context of linear relations it is
\emph{generically}
the case that the Douglas--Rachford operator is only a resolvent
and \emph{not} a proximal mapping. 
}

The rest of the paper is organized as follows. 
In Section~\ref{s:aux}, we develop auxiliary results on matrices,
proximal mappings and convergence. 
Section~\ref{s:main} contains our main result. 

Finally, notation and notions not explicitly defined may 
be found in, e.g., \cite{BC2011}, \cite{MN}, 
\cite{Rock70},  or \cite{Rock98}.

\section{Auxiliary results}

\label{s:aux}

\subsection{Matrices}

Unless stated otherwise, we view $\RR^{n\times n}$, the set of
real $n\times n$ matrices, as a Banach space,
with norm $\|R\| := \sup_{\|x\|\leq 1} \|Rx\|$, which is the
square root of the largest
eigenvalue of $R\tran R$.
We denote by $\SR^n$ the subspace of symmetric $n\times n$
matrices.
A matrix $R$ is \emph{nonexpansive} if $\|R\|\leq 1$, i.e.,
$R$ belongs to the unit ball of $\RR^{n\times n}$. 
This set is convex, closed, and has $0$ in its interior.
The set of nonexpansive symmetric matrices is likewise in $\SR^n$. 

\begin{lemma}
\label{l:key}
Let $R_0,S_0,R_1,S_1$ be matrices in $\RR^{n\times n}$.
Suppose that $R_0$ commutes with $S_0$,
but that $R_1$ does not commute with $S_1$. 
For each $\lambda\in\zeroun$,
set $R_\lambda = (1-\lambda)R_0+\lambda R_1$ and 
set $S_\lambda = (1-\lambda)S_0+\lambda S_1$. 
Then $\menge{\lambda\in\zeroun}{\text{$R_\lambda$ commutes with
$S_\lambda$}}$ is either empty or a singleton.  
\end{lemma}
\begin{proof}
For $\lambda\in[0,1]$, consider
the matrix
\begin{equation}
M_\lambda = R_\lambda S_\lambda - S_\lambda R_\lambda.
\end{equation}
By hypothesis, $M_0=0$ but $M_1\neq 0$. 
Since $M_1\neq 0$, there exist $(i,j)\in\{1,\ldots,n\}^2$ such
that the $(i,j)$ entry of $M_1$ is not $0$. 
Denote by $q(\lambda)$ the $(i,j)$ entry of $M_\lambda$.
Then $q(\lambda)$ is a polynomial in $\lambda$ of degree at most $2$,
with $q(0)=0$ and $q(1)\neq 0$. 
On $\zeroun$, $q$ has at most one root.
Therefore, with the possible exception of one value
$\lambda\in\zeroun$, $M_\lambda \neq 0$.
\end{proof}

\begin{example}
\label{ex:coolmat}
Suppose that $n\geq 2$ and define matrices in $\SR^n$ by 
\begin{equation}
\sbox0{$\begin{matrix}1&0\\0&-1\end{matrix}$}
R_0=R_1=\left[
\begin{array}{c|c}
\usebox{0}&\makebox[\wd0]{\large $0$}\\
\hline
  \vphantom{\usebox{0}}\makebox[\wd0]{\large $0$}&\makebox[\wd0]{\large $0$}
\end{array}
\right], 
\quad
\sbox0{$\begin{matrix}-1&0\\0&1\end{matrix}$}
S_0=\left[
\begin{array}{c|c}
\usebox{0}&\makebox[\wd0]{\large $0$}\\
\hline
  \vphantom{\usebox{0}}\makebox[\wd0]{\large $0$}&\makebox[\wd0]{\large $0$}
\end{array}
\right], 
\quad
\sbox0{$\begin{matrix}0&1\\1&0\end{matrix}$}
S_1=\left[
\begin{array}{c|c}
\usebox{0}&\makebox[\wd0]{\large $0$}\\
\hline
  \vphantom{\usebox{0}}\makebox[\wd0]{\large $0$}&\makebox[\wd0]{\large $0$}
\end{array}
\right]. 
\end{equation}
For each $\lambda\in\zeroun$, set
$R_\lambda = (1-\lambda)R_0+\lambda R_1$ and 
set $S_\lambda = (1-\lambda)S_0+\lambda S_1$. 
Let $\lambda\in[0,1]$. 
Then 
\begin{equation}
\|R_\lambda\|=1, \;
\|S_\lambda\|=\textstyle \sqrt{(1-\lambda)^2+\lambda^2}\in
\big[1/\sqrt{2},1\big]
\end{equation}
and
\begin{equation}
\sbox0{$\begin{matrix}\lambda-1&\lambda\\-\lambda&\lambda-1\end{matrix}$}
R_\lambda S_\lambda =\left[
\begin{array}{c|c}
\usebox{0}&\makebox[\wd0]{\large $0$}\\
\hline
  \vphantom{\usebox{0}}\makebox[\wd0]{\large $0$}&\makebox[\wd0]{\large $0$}
\end{array}
\right],\;
\sbox0{$\begin{matrix}\lambda-1&-\lambda\\\lambda&\lambda-1\end{matrix}$}
S_\lambda R_\lambda =\left[
\begin{array}{c|c}
\usebox{0}&\makebox[\wd0]{\large $0$}\\
\hline
  \vphantom{\usebox{0}}\makebox[\wd0]{\large $0$}&\makebox[\wd0]{\large $0$}
\end{array}
\right].
\end{equation}
Consequently, $R_\lambda$ commutes with $S_\lambda$ if and only
if $\lambda=0$. 
\end{example}

\subsection{Proximal mappings}

We now characterize proximal mappings within the set of
resolvents. 

\begin{lemma} \cite[Lemma~4.36]{Schaad}
Let $T\in\RR^{n\times n}$ be a proximal mapping. 
Then $T=T\tran$.
\end{lemma}
\begin{proof}
Set $q\colon x\mapsto \tfrac{1}{2}\|x\|^2$.
Then $\Id = \nabla q$.
By hypothesis, $T$ is a proximal mapping, so 
there exists a convex $f$ such that 
\begin{equation}
T = (\Id +\partial f)^{-1} = (\partial (q+f))^{-1} = \partial
(q+f)^* = \nabla (q+f)^*. 
\end{equation}
It follows that $T = \nabla T = \nabla^2(q+f)^*$ is symmetric. 
\end{proof}

It turns out that the converse of the previous result also holds. 

\begin{lemma} 
Let $T\in\RR^{n\times n}$ be firmly nonexpansive\footnote{For
further information on firmly nonexpansive mappings, see
\cite{BC2011} and \cite{GR}.} and such that 
$T=T\tran$. Then $T$ is a proximal mapping.
\end{lemma}
\begin{proof}
Set $f\colon X\to\RR\colon x\mapsto \tfrac{1}{2}\scal{x}{Tx}$.
Since $T$ is symmetric, we have $\nabla f = T$. 
Since $T$ is firmly nonexpansive, it is monotone and thus $f$ is
convex. 
By the (extended form of the) Baillon--Haddad theorem
(see \cite[Theorem~18.15]{BC2011}),
$\nabla f = T$ is a proximal map. 
\end{proof}

We thus obtain the following useful characterization of
proximal mappings. 

\begin{corollary}
\label{c:key}
Let $T\in\RR^{n\times n}$. 
Then $T$ is a proximal mapping if and only if
$T$ is both firmly nonexpansive and symmetric. 
\end{corollary}

\subsection{Convergence}

From now on, we denote 
the set of maximally monotone operators on $X$ by $\mM$,
the subset of linear relations\footnote{A linear relation on $X$ is
set-valued map from $X$ to $X$ such that its graph is a linear
subspace of $X\times X$. In relationship to the present paper, 
we refer the reader to \cite{63} for more on maximally monotone
linear relations. Furthermore, a resolvent $J_A$ is linear if and only
if $A\in\mL$ by \cite[Theorem~2.1(xviii)]{65}.} by $\mL$,
and the subdifferential operators of functions in
$\Gamma_0(X)$ by $\mS$. 

Let $(A_k)_\kkk$ be a sequence in $\mM$ and let $A\in\mM$.
Then $(A_k)_\kkk$ converges to $A$ \emph{graphically}, in symbols 
$A_k \gconv A$ if and only if the resolvents converge pointwise,
in symbols, 
$J_{A_k}\pconv J_A$. 
This induces a metric topology on $\mM$
(see \cite{Rock98} for details). 
Note that $\mL$ is a closed topological subspace of $\mM$
and that pointwise convergence by resolvents can in that setting 
be replaced by
convergence in operator norm (since $X$ is finite-dimensional). 

The following result is now easily verified. 

\begin{proposition}
\label{p:conv}
Let 
 $(A_k)_\kkk$ be a sequence in $\mL$, and let $A\in\mL$.
Then we have the equivalences
\begin{equation}
A_k\gconv A
\Leftrightarrow
J_{A_k}\pconv J_A
\Leftrightarrow
R_{A_k}\pconv R_A
\Leftrightarrow
J_{A_k}\to J_A
\Leftrightarrow
R_{A_k}\to R_A.
\end{equation}
\end{proposition}

We thus are able to define a metric on $\mL$ by
\begin{equation}
(A_1,A_2)\mapsto  
\|J_{A_1}-J_{A_2}\|
\end{equation}
and a metric on $\mL\times\mL$ by
\begin{equation}
\big((A_1,B_1),(A_2,B_2)\big)\mapsto 
\|J_{A_1}-J_{A_2}\| + \|J_{B_1}-J_{B_2}\|.
\end{equation}
Note that in view of the pointwise characterization of
Proposition~\ref{p:conv}, both $\mL$ and $\mL\times\mL$ are
\emph{complete} and so are $\mL\cap\mS$ and
$(\mL\cap\mS)\times(\mL\cap\mS)$. 

These topological notions are used in the next section which
contains our main result.

\section{Main result}

\label{s:main}

Recall that for $A$ and $B$ in $\mM$, the Douglas--Rachford
operator is defined by 
\begin{equation}
T(A,B) = T_{(A,B)} = \tfrac{1}{2}\big(\Id + R_BR_A\big).
\end{equation}
Note that $T_{(A,B)}$ is firmly nonexpansive
and the resolvent of some maximally monotone operator
$M(A,B)\in\mM$ but it may be the case that $M(A,B)\notin\mS$ even
when $A$ and $B$ belong to $\mS$.

We are ready for our main result.

\begin{theorem}
\label{t:main}
Suppose that $n\geq 2$. 
Then generically, the Douglas--Rachford operators for
symmetric linear relations are not proximal mappings;
in fact, the set
\begin{equation}
D := \menge{(A,B)\in(\mL\cap \mS)^2}{T_{(A,B)}\text{\rm \ is a proximal
map}}
\end{equation}
is a closed subset of $(\mL\cap\mS)^2$ that is nowhere dense. 
\end{theorem}
\begin{proof}
We start by verifying that $D$ is closed. To this end, 
let $(A_k,B_k)_\kkk$ be a sequence in $D$ converging
to $(A,B)\in(\mL\cap\mS)^2$. 
By definition, $T(A_k,B_k)$ is a proximal mapping for every
$\kkk$. By Corollary~\ref{c:key}, $(\forall \kkk)$
$T(A_k,B_k)\tran = T(A_k,B_k)$. 
Hence $(\forall\kkk)$ $R_{A_k}R_{B_k}=R_{B_k}R_{A_k}$.
In view of Proposition~\ref{p:conv}, we take the limit and obtain
$R_{A}R_{B}=R_{B}R_{A}$. Thus $T(A,B)=T(A,B)\tran$.
Using Corollary~\ref{c:key} again, we deduce that $T(A,B)$ is a
proximal map. 

We now show that $D$ is nowhere dense.
Let $(A_0,B_0)$ be in $D$. 
Then $R_{A_0}R_{B_0}=R_{B_0}R_{A_0}$.
Next, set $A_1 = ((R_1+\Id)/2)^{-1}-\Id$ and 
$B_1 = ((S_1+\Id)/2)^{-1}-\Id$, where $R_1$ and $S_1$ are
as in Example~\ref{ex:coolmat}. 
Then $R_{A_1}=R_1$ and $R_{B_1}=S_1$ do not commute
and hence $(A_1,B_1)\notin D$. 
Now set
\begin{subequations}
\begin{equation}
\big(\forall\lambda\in\zeroun\big)\quad
A_\lambda = \Big(\tfrac{1}{2}\Id +
\tfrac{1}{2}\big((1-\lambda)R_{A_0}+\lambda R_{A_1}\big) \Big)^{-1}-\Id 
\end{equation}
and 
\begin{equation}
\big(\forall\lambda\in\zeroun\big)\quad
B_\lambda = \Big(\tfrac{1}{2}\Id +
\tfrac{1}{2}\big((1-\lambda)R_{B_0}+\lambda R_{B_1}\big)
\Big)^{-1}-\Id .
\end{equation}
\end{subequations}
Then, as $\lambda\to 0^+$,
\begin{equation}
R_{A_\lambda} = (1-\lambda)R_{A_0}+\lambda R_{A_1}\to R_{A_0}
\text{~and~}
R_{B_\lambda} = (1-\lambda)R_{B_0}+\lambda R_{B_1}\to R_{B_0}.
\end{equation}
It follows from Lemma~\ref{l:key} and Corollary~\ref{c:key}
that there exists $\mu\in\left]0,1\right]$ such that
$(\forall \lambda\in\left]0,\mu\right])$
$(A_\lambda,B_\lambda)\notin D$.
Hence $(A_0,B_0)$ does not belong to the interior of $D$. 
\end{proof}

\begin{remark}
The assumption that $n\geq 2$ in Theorem~\ref{t:main} is
important: indeed, when $n=1$, it is well known that every maximally monotone
operator is actually a subdifferential operator (see, e.g.,
\cite[Corollary~22.19]{BC2011}) and therefore every
Douglas--Rachford operator is a proximal mapping in this case. 
\end{remark}

\begin{remark}[open problems]
The following questions appear to be of interest:
\begin{enumerate}
\item 
Does Theorem~\ref{t:main}
admit an extension from symmetric linear relations to general
subdifferential operators?
\item The set $D$ in Theorem~\ref{t:main} is closed and nowhere
dense. Is it also a porous\footnote{See \cite{RZ} for further
information on porous sets.} set?
\end{enumerate}
\end{remark}

\section*{Acknowledgements}
HHB was partially supported by the Natural Sciences and Engineering Research Council of Canada and by the Canada Research Chair Program. 
XW was partially supported by the Natural Sciences and
Engineering Research Council of Canada.


\end{document}